\documentclass[12pt]{article}
\usepackage[latin1]{inputenc}
\usepackage{amsmath}
\usepackage{amsfonts}
\usepackage{amssymb}
\usepackage{booktabs}
\usepackage{multirow}
\usepackage[american]{babel}
\usepackage{graphicx}
\usepackage{amsthm}
\usepackage{color}
\usepackage{xcolor}
\usepackage[normalem]{ulem}
\usepackage[text={5in,7.125in},centering]{geometry}
\usepackage{hyperref}
\usepackage{caption}
\usepackage{float}

\begin{document}


\title{A lower bound for the complex flow number of a graph: a geometric approach.}{}

\newtheorem{theorem}{Theorem}
\newtheorem{proposition}[theorem]{Proposition}
\newtheorem{lemma}[theorem]{Lemma}
\newtheorem{definition}[theorem]{Definition}
\newtheorem{example}[theorem]{Example}
\newtheorem{corollary}[theorem]{Corollary}
\newtheorem{conjecture}[theorem]{Conjecture}
\newtheorem{remark}{Remark}
\newtheorem{problem}{Problem}
\newtheorem{claim}{Claim}

\author{Davide Mattiolo \footnote{Department of Computer Science, KU Leuven Kulak, 8500 Kortrijk, Belgium}, Giuseppe Mazzuoccolo \footnote{Dipartimento di Informatica, Universit\`{a} degli Studi di Verona, Italy},\\Jozef Rajn\'{i}k \footnote{Department of Computer Science, Comenius University in Bratislava, Slovakia}, Gloria Tabarelli \footnote{Dipartimento di Matematica, Universit\`{a} di Trento,
Italy}}

\vspace*{-2cm}
{\let\newpage\relax\maketitle}

\begin{abstract}

Let $r \geq 2$ be a real number.  A complex nowhere-zero $r$-flow on a graph $G$ is an orientation of $G$ together with an assignment $\varphi\colon E(G)\to  \mathbb{C}$ such that, for all $e \in E(G)$,
the modulus of the complex number $\varphi(e)$ lies in the interval $[1,r-1]$ and, for every vertex, the incoming flow is equal to the outgoing flow. The complex flow number of a bridgeless graph $G$, denoted by $\phi_{\mathbb{C}}(G)$, is the minimum of the real numbers $r$ such that $G$ admits a complex nowhere-zero $r$-flow.
The exact computation of $\phi_{\mathbb{C}}$ seems to be a hard task even for very small and symmetric graphs. 
In particular, the exact value of $\phi_{\mathbb{C}}$ is known only for families of graphs where a lower bound can be trivially proved.
Here, we use geometric and combinatorial arguments to  give a non trivial lower bound for $\phi_{\mathbb{C}}(G)$ in terms of the odd-girth of a cubic graph $G$ (i.e. the length of a shortest odd cycle) and we show that such lower bounds are tight. Our main result, Theorem \ref{cor:oddgirth}, relies on the exact computation of the complex flow number of the wheel graph $W_n$ (see Theorem \ref{thm:wheels}). In particular, we show that for every odd $n$, the value of $\phi_{\mathbb{C}}(W_n)$ arises from one of three suitable configurations of points in the complex plane according to the congruence of $n$ modulo $6$.\end{abstract}

\section{Introduction}\label{section intro}
The theory of integer nowhere-zero flows on finite graphs represents a very active research area in graph theory (see for example \cite{Exp_many_nzf}, \cite{Jaeger}, \cite{kochol_5-flow_girth_11}, \cite{Lov_thomassen_Zhang:3NZF}, \cite{MS_5_flow_oddness4}, \cite{Seymour6flow}, \cite{TUTTE_contr_chrom_pol}, \cite{ZhangBook}). The generalization to real numbers is also well-studied (see for instance \cite{EMT16}, \cite{tarsizhang}, \cite{GMM_cfn5}, \cite{LuSk}, \cite{MS_edge_col_cflows}), while very few is known in the complex case or, more in general, for flows taking values in $\mathbb{R}^d$ (see \cite{Tho}, \cite{Zhangandal}, \cite{DeVos}).
Let $r \geq 2$ be a real number.  A $d$-dimensional nowhere-zero $r$-flow on a graph $G$, an $(r, d)$-NZF on $G$ from now on, is an orientation of $G$ together with an assignment $\varphi\colon E(G)\to  \mathbb{R}^d$ such that, for all $e \in E(G)$,
the (Euclidean) norm of $\varphi(e)$ lies in the interval $[1,r-1]$ and, for every vertex, the sum of the inflow and outflow is the zero element in $\mathbb{R}^d$. The $d$-dimensional  flow number of a bridgeless graph $G$, denoted by $\phi_d(G)$, is defined as the infimum of the real numbers $r$ such that $G$ admits an $(r, d)$-NZF.

In this paper, we consider only the case $d=2$. For this reason, in order to simplify the notation, we refer to an $(r,2)$-NZF on a bridgeless graph $G$ as a \emph{complex nowhere-zero $r$-flow} on $G$, and to its $2$-dimensional flow number $\phi_2(G)$ as its \emph{complex flow number}, denoting it from now on by $\phi_\mathbb{C}(G)$.

It can be easily proved (see \cite{d_dim_flow}) that $\phi_\mathbb{C}(G)$ (and more in general $\phi_d(G)$) is actually a minimum. So, given a bridgeless graph $G$, it always admits a complex $\phi_\mathbb{C}(G)$-flow, not necessarily unique. We will refer to a complex $\phi_\mathbb{C}(G)$-flow as an \emph{optimal complex flow} of $G$.

A general upper bound for $\phi_\mathbb{C}(G)$, where $G$ is a bridgeless graph, has been proposed by the authors in \cite{d_dim_flow}. However, the exact value of $\phi_\mathbb{C}(G)$ is known only when $G$ belongs to very specific classes of graphs. For all these graphs, a lower bound for $\phi_\mathbb{C}(G)$ can be easily proved, since either it is the minimum admissible value (i.e. $\phi_\mathbb{C}(G)=2$) or since it easily arises from some specific local properties of the graph $G$.
Establishing good lower bounds for $\phi_\mathbb{C}(G)$ remains in general the hardest task in the study of such a parameter, even if we focus on the class of cubic graphs. Note that the restriction to the class of cubic graphs is standard in flow theory and it can be applied to complex flows as well.

In this paper, we use a combination of geometric and combinatorial arguments to prove a non-trivial lower bound for $\phi_\mathbb{C}(G)$, see Corollary \ref{cor:oddgirth}, in terms of the length of a shortest odd cycle of a bridgeless cubic graph $G$. Theorem \ref{cor:oddgirth} is a straightforward consequence of Theorem \ref{thm:wheels}, where the complex flow number is exactly determined for every wheel graph $W_n$ of order $n+1$. The proof of Theorem \ref{thm:wheels} shows that there exists an optimal complex flow of $W_n$ which can be described by one of three suitable sequences of points in the complex plane, according to the congruence of $n$ modulo $6$. 


\section{Complex flow number of $W_n$}


For every integer $n\geq 3$, let $W_n$ be the wheel graph with $n+1$ vertices and consider the orientation of its edges as in Figure \ref{fig:W5}. More precisely, the $n$ vertices of the external cycle of $W_n$ are labeled with $v_0,v_1,...,v_{n-1}$ and the central vertex with $u$. All edges $uv_j$ and $v_{j-1}v_j$ in the chosen orientation of $W_n$ are directed towards $v_j$ (here and in what follows indices are taken modulo $n$). 

Let $\varphi$ be a $(\lambda+1,2)$-NZF of $W_n$. Set
	$$	
	\varphi(uv_j)=z_j \in \mathbb{C}, j \in \{0,...,n-1\},$$
	$$
	\varphi(v_jv_{j+1})=p_j \in \mathbb{C}, j \in \{0,...,n-1\}.
	$$

Along the paper, we will refer to the elements $z_j$ and $p_j$ simply as  complex numbers, all operations performed on them will be assumed to be standard operations on complex numbers. In particular, since $\varphi$ is a $(\lambda+1,2)$-NZF of $W_n$, the modulo of each flow value is a real number which lies in the interval $[1,\lambda]$, i.e. $1\leq |p_j|, |z_j| \leq \lambda$ holds. 
Moreover, the relation \begin{equation}\label{zi} z_j=p_j-p_{j-1} \end{equation} holds for every $j=0,...,n-1$. 
Relation (\ref{zi}) suggests that the knowledge of all points $p_j$ is sufficient to reconstruct the flow. Then, a natural representation of the flow is a sequence of $n$ points of the complex plane. For this reason, we will refer to the elements $p_j$ as points. Moreover, we associate to such a sequence a list of vectors, denoted by $p_{j-1}p_j$ and directed from $p_{j-1}$ to $p_j$. Very often, with a slight abuse of terminology, we refer to $p_{j-1}p_{j}$ as the vector $z_j$.    

Hence, along the paper we always represent a complex flow of $W_n$ as a cyclic sequence (i.e. the first element of the sequence is considered to succeed the last one) of $n$ points $(p_0,\dots,p_{n-1})$ in the complex plane such that all points belong to the circular crown between circumferences centered in the origin and of radius $1$ and $\lambda$, denoted by $\mathcal{C}_I$ and $\mathcal{C}_E$ respectively, and with the additional property that the norm of all vectors $p_{j-1}p_j$ lies also in the interval $[1,\lambda]$ (see for instance Figure \ref{fig:W5}). 
	
    \begin{figure}
	\center
	\includegraphics[]{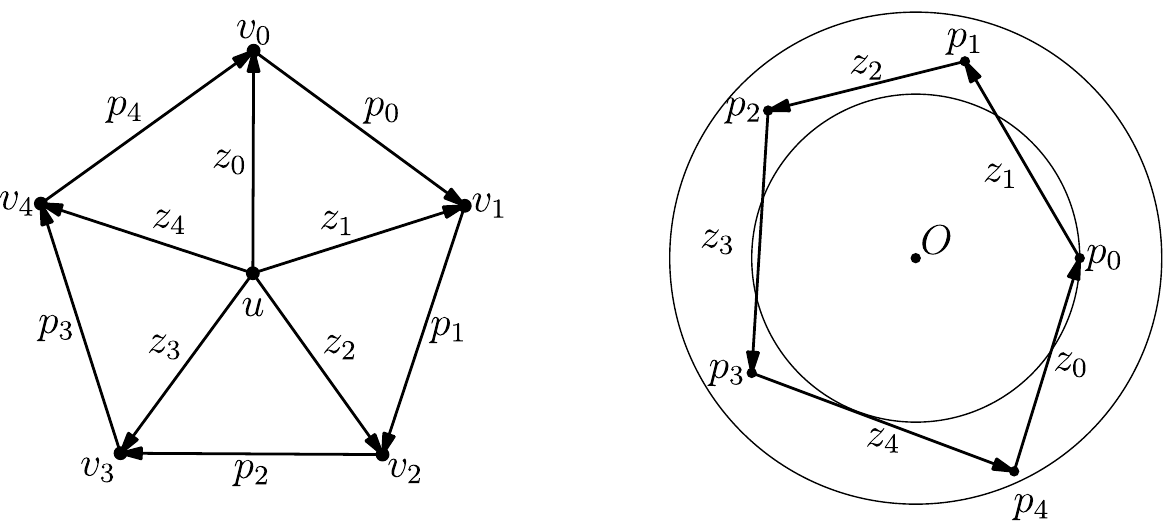}
	\caption{A representation of a complex flow of $W_5$.}
	\label{fig:W5}
	\end{figure}

By using such a representation, we exhibit a complex flow of $W_n$ for each odd $n$. Then, we prove its optimality in Theorem \ref{thm:wheels}.

Set $t=\lfloor \frac{n}{6} \rfloor$ for every odd $n$. We distinguish three cases according to the congruence of $n$ modulo $6$. We furnish a geometric description of each case and then we formally give the sequences of points representing the flows.

For $n \equiv 5 \pmod 6$, we consider points $p_j$ as the vertices of a regular star polygon $\{\frac{n}{t+1}\}$ (following the standard Schl\"afli notation, see \cite{CoxeterBook}) inscribed in $\mathcal{C}_I$. The length of each side of the polygon is equal to $2\sin(\frac{\pi}{6}\cdot \frac{n+1}{n})$. 
For $n \equiv 1 \pmod 6$, we construct points $p_j$ on $\mathcal{C}_I$ as follows: starting from $p_0$ and moving in clockwise direction, we have $p_1$ at distance $1$ from $p_0$. All others points are obtained by moving on $\mathcal{C}_I$ in anticlockwise direction, each point at distance $2\sin(\frac{\pi}{6}\cdot\frac{n}{n-1})$ from the previous one. The distance between $p_{n-1}$ and $p_0$ results to be also $2\sin(\frac{\pi}{6}\cdot \frac{n}{n-1})$.
For $n \equiv 3 \pmod 6$, we construct points $p_j$ on $\mathcal{C}_I$ except $p_1$ which belongs to $\mathcal{C}_E$. Starting from $p_0 \in \mathcal{C}_I$ and moving in clockwise direction, we have $p_1 \in \mathcal{C}_E$ at distance $1$ from $p_0$. Then, $p_2 \in \mathcal{C}_I$ is at distance $1$ from $p_1$, again in clockwise direction. All others points are obtained following $\mathcal{C}_I$ in anticlockwise direction, each at distance $2\sin(\frac{\pi}{6}\cdot \frac{n}{n-1})$ from the previous one. Once again also the distance between $p_{n-1}$ and $p_0$ is $2\sin(\frac{\pi}{6}\cdot\frac{n}{n-1})$. Figure \ref{fig:optimal_flows} represents an example of the described sequences for each possible odd congruence class modulo $6$.

Hence, the sequences of points result to be the following.

\begin{itemize}
 \item if $n \equiv 5 \pmod 6$, 
 $$p_j=e^{ij\left(\frac{\pi}{3}\cdot \frac{n+1}{n}\right)},  \forall j : 0\leq j \leq n-1,$$
\item if $n \equiv 1 \pmod 6$, 
$$p_0=e^{i\frac{\pi}{3}} \text{ and }   
 p_{j+1}=e^{ij\left(\frac{\pi}{3}\cdot\frac{n}{n-1}\right)}, \forall j: 0\leq j \leq n-2,$$
 \item if $n \equiv 3 \pmod 6$,  
 $$ p_0=e^{2i\left(\frac{\pi}{6}\cdot\frac{2n-3}{n-1}\right)}, p_1=2\sin\left(\frac{\pi}{6}\cdot\frac{n}{n-1}\right)e^{i\left(\frac{\pi}{6}\cdot\frac{2n-3}{n-1}\right)} \text{ and } $$ $$p_{j+2}=e^{ij\left(\frac{\pi}{3}\cdot\frac{n}{n-1}\right)}, \forall j: 0\leq j \leq n-3.$$
\end{itemize}

For each $n$, we denote by $\lambda$ the maximum distance between two consecutive points of the corresponding sequence. Note that in the case $n \equiv3 \pmod 6$, such a value is also equal to $|p_1|$. We have a $(\lambda+1,2)$-NZF of $W_n$ and $\lambda + 1$ gives an upper bound for $\phi_{\mathbb{C}}(W_n)$. 
\begin{remark}\label{remark:W3}
For each $n \geq 3$, $\phi_{\mathbb{C}}(W_n)\leq\phi_{\mathbb{C}}(W_3) \leq 1+\sqrt{2}$. 
\end{remark}
We will make use of this remark along the proof of Theorem \ref{thm:wheels} in order to guarantee a general upper bound for $\phi_{\mathbb{C}}(W_n)$ which will be sufficiently small for our aims. 


\begin{figure}
	\center
	\includegraphics[]{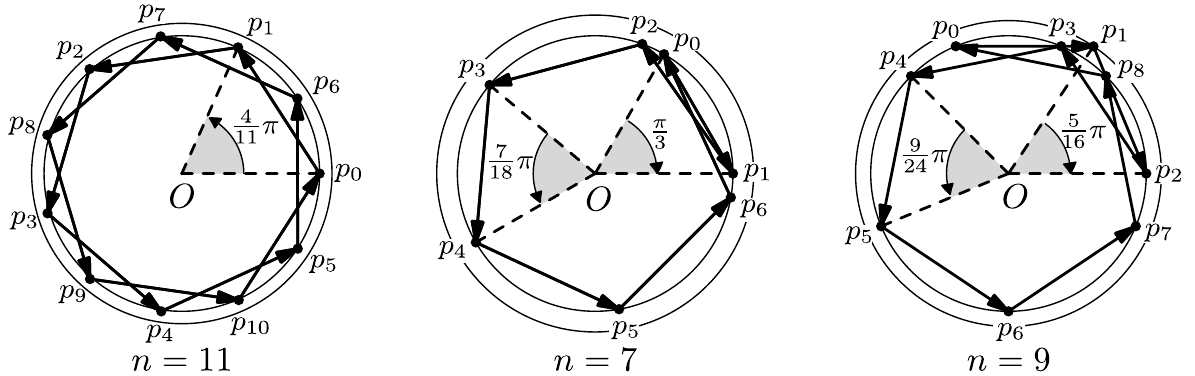}
	\caption{Three sequences corresponding to, from left to right, optimal flows of $W_{11}$, $W_7$ and $W_9$.}
	\label{fig:optimal_flows}
	\end{figure}

\begin{theorem} \label{thm:wheels}
	Let $W_n$ be the wheel graph of order $n+1$, for $n \geq 3$. Then, $$\phi_{\mathbb{C}}(W_n) = \begin{cases}
	2& \text{ if } n \text{ is even}, \\
	1 + 2\sin(\frac{\pi}{6}\cdot\frac{n}{n-1})& \text{ if } n \equiv 1,3\mod 6, \\
	1 + 2\sin(\frac{\pi}{6}\cdot\frac{n+1}{n})& \text{ if } n\equiv 5 \mod 6. \\                                                                                                                               \end{cases}
	$$
\end{theorem}

\begin{proof}
	If $n$ is even then $W_n$ has a $(3,1)$-NZF. Therefore, by Proposition 1 in \cite{Tho}, $\phi_{\mathbb{C}}(W_n)=2.$
	
	Let $n$ be odd and let $\varphi$ be an optimal $(\lambda+1,2)$-NZF of $W_n$. From now on we can assume $\lambda \leq \sqrt{2}$ due to Remark \ref{remark:W3}. Let $(p_0,\dots,p_{n-1})$ be the associated cyclic sequence of points $p_j$. As already remarked, we denote the vector $p_{j-1}p_{j}$ by $z_j$ for each $j \in \{0,\dots,n-1\}$, where all indices are taken modulo $n$.  In particular, we have $\max_j\{|z_j|, |p_j|\} =\lambda$ since $\varphi$ is optimal. 
	If $\theta \in (-\pi,\pi)$ denotes the amplitude of an angle and $\theta >0$ ($\theta < 0$), then the positive (negative) rotation is by definition in anticlockwise (clockwise) direction. Similarly, if $p_{j-1}=|p_{j-1}|e^{i\alpha_{j-1}}$ and $p_j=|p_j|e^{i\alpha_j}$ are two consecutive points in the cyclic sequence then the vector $z_j=p_{j-1}p_j$ is said to be positively (negatively) oriented, or simply positive (negative), if $\alpha_j-\alpha_{j-1}$ is positive (negative). Note that $\alpha_{j-1} \neq \alpha_{j}$ since $\lambda \leq \sqrt{2}$ $(< 2)$.

 First of all, let us define some geometric transformations of the cyclic sequence $(p_0,\dots,p_{n-1})$ that will be largely used in what follows.
 Let $\theta \in \mathbb{R}$ and $h,k \in \{0,\dots,n-1\}$. Define $\rho_{h,k}(\theta)$ as the transformation which rotates all points $p_h, p_{h+1},\dots, p_k$ around the origin by an angle of $\theta$ and fixes all the others. 
	Note that we are considering a cyclic sequence, hence $k$ could be less than $h$ and $\rho_{h,k}(\theta) \neq \rho_{k,h}(\theta)$. Indeed, if $h<k$, we have 
	
	\begin{multline*}\rho_{h,k}(\theta)(p_0,\dots,p_h,\dots,p_k,\dots,p_{n-1}) =\\
	=(p_0,\dots,p_{h-1},p_he^{i\theta},\dots,p_ke^{i\theta},p_{k+1},\dots,p_{n-1}),
	\end{multline*}
	and 
	\begin{multline*}\rho_{k,h}(\theta)(p_0,\dots,p_h,\dots,p_k,\dots,p_{n-1}) =\\
	= (p_0e^{i\theta},\dots,p_he^{i\theta},p_{h+1},\dots,p_{k-1},p_ke^{i\theta},\dots,p_{n-1}e^{i\theta}).\end{multline*}
	
	\begin{claim}
	\label{remark:transformations}
	For any $h \ne k \in \{0, 1, \dots, n - 1\}$ there is an angle $\theta$ such that the sequence $\rho_{h,k-1}(\theta)(p_0, \dots, p_{n-1})=(p'_0, \dots, p'_{n-1})$ satisfies:
	\begin{itemize}
	\item[(a)] $|p_j|=|p'_j|$ for all $j$, $|z_j|=|z'_j|$ for all $j\notin \{h,k\}$;
	\item[(b)] $\lambda>|z'_h|>|z_h|$ and $1<|z'_k|<|z_k|$, if $z_h$ and $z_k$ have the same orientation. In this case, we will say that the transformations lengthens $z_h$ and shortens $z_k$ (by an arbitrary small factor).
	\item[(c)] either $\lambda>|z'_h|>|z_h|$ and $\lambda>|z'_k|>|z_k|$, or $1<|z'_h|<|z_h|$ and $1<|z'_k|<|z_k|$, if $z_h$ and $z_k$ have opposite orientations. In the former (latter) case we say that the trasformation lengthens (shorthens) both of them (by an arbitrary small factor).
	\end{itemize}
	\end{claim}
	
	\begin{proof}
	Assertion (a) directly follows by the definition of $\rho_{h,k-1}(\theta)$.
	For Case (b) we choose $\theta$ positive or negative according to the common orientation of $z_h$ and $z_k$. In Case (c), say that $z_h$ is positive and $z_k$ is negative, we choose $\theta$ positive (negative) if we want to legthen (shorten) the vectors $z_h$ and $z_k$. In all the cases the absolute value of $\theta$ can be chosen arbitrary small to ensure arbitrary small scale factor and then $1<|z'_j|<\lambda$ for every $j$.
	\end{proof}
	
	For our aims, we also need to define $\sigma_{h,k}(\theta)$ as the transformation which rotates the point $p_h$ around the point $p_k$ by an angle $\theta$ and fixes any other point of the sequence.
	$$\sigma_{h,k}(\theta)(p_0,\dots,p_h,\dots,p_{n-1}) = (p_0,\dots,(p_h-p_k)e^{i\theta}+p_k,\dots,p_{n-1}).$$

    The main idea of the proof is choosing time by time an optimal flow $\varphi$ of $W_n$ satisfying additional minimality assumptions (explained later in details). 
    We will show that if such a $\varphi$ does not correspond to one of the three sequences (up to isometries) in Figure \ref{fig:optimal_flows}, then we can modify it to obtain a new sequence which contradicts the minimality assumptions on $\varphi$.	
	
	\begin{itemize}

	\item \textbf{Assume $|z_j|<\lambda$ for every $j \in \{0,1,...,n-1\}$};
	
	Consider the optimal 2-dimensional flows of $W_n$ having the minimum number, say $m_1$, of values $z_j$ with $|z_j|=1$. Among them, choose $\varphi$ with the minimum number, say $m_2$, of points $p_j$ with $|p_j|=\lambda$ (i.e. $p_j \in \mathcal{C}_E$). 
	Moreover, without loss of generality, we can assume that $\varphi$ has at least one of the vectors $z_j$ which is positive, otherwise we can simply consider $-\varphi$.	
	
	First of all, we prove that by our choice of $\varphi$ the relation $|z_j|=1$ follows for every index $j$ and that all vectors $z_j$ are positive.
	
	Suppose by contradiction that there exists an index $h$ such that $|z_h|>1$. By assumption $|z_h|<\lambda$. 
	If $m_1>0$, then there exists $k$ such that $|z_k|=1$.
	According to Claim \ref{remark:transformations}, we can lenghten $z_k$ and shorten or lengthen $z_h$ (according to its orientation) constructing a sequence of points having less than $m_1$ vectors of modulo $1$, a contradiction. 
	Then, we can assume $m_1=0$.
	Note that since $\varphi$ is optimal, there exists $k\in \{0,1,...,n-1\}$ such that $p_k \in \mathcal{C}_E$. 
	Construct a new sequence by multiplying $p_k$ by a factor $1-\varepsilon$, and leaving invariant all other points. If $\varepsilon>0$ is sufficiently small, the new sequence has no vectors of modulo $1$ like the original sequence, but less than $m_2$ points belonging to $\mathcal{C}_E$, a contradiction again.
	Up to now we have that $|z_j|=1$, for every index $j$.

	Assume there exist two indices $h,k$ such that $z_h$ is positive and $z_k$ is negative.
	Applying Claim \ref{remark:transformations} we can lengthen both of them to obtain less than $m_1$ vectors having modulo $1$, a contradiction with the choice of $\varphi$.
	
	Then, all vectors $z_j$ are positive and with $|z_j|=1$. Now we show that for each odd $n$ a sequence of points $p_j$ with such properties corresponds to a $(\lambda+1,2)$-flow having $\lambda$ larger than the value in the statement. This leads to a contradiction since $\varphi$ is chosen to be optimal.
	Indeed, let $\alpha_j>0$ be the angle subtended by the vector $z_j=p_{j-1}p_j$. It holds that $\sum_{j=0}^{n-1} \alpha_j=2a\pi$ for some positive integer $a$. Moreover, since $\lambda\leq \sqrt{2}<\Phi$ (Golden Ratio) the angle $\alpha_j$ is at least $2\arcsin(\frac{1}{2\lambda})$ which is the angle obtained with $p_{j-1},p_j \in \mathcal{C}_E$. Hence, we have $$\arcsin\left(\frac{1}{2\lambda}\right)\leq \frac{a}{n}\pi.$$	
	
	We look for the mimimum possible $\lambda$ which realizes previous inequality. It is clearly obtained when the equality holds. Moreover, since $\lambda>1$, $\arcsin(\frac{1}{2\lambda}) <\frac{\pi}{6}$ holds, that is $a <\frac{n}{6}$. So, $\lambda$ is minimum and larger than $1$ for $a=\lfloor\frac{n}{6}\rfloor$. Hence, if $n=6t+h$, we have that $\arcsin(\frac{1}{2\lambda})=\frac{t}{n}\pi$ and so $\lambda=\frac{1}{2\sin(\frac{t}{n}\pi)}$. Direct computations show that, for every odd $n$, this value of $\lambda$ is greater than the corresponding value in the statement of the theorem.
\end{itemize}

	\begin{itemize}
	
	\item\textbf{Assume $\exists k \in \{0,...,n-1\}$ such that $|z_k|=\lambda$};
	
	Without loss of generality we can assume $z_k$ positive.
	Consider the set of optimal complex flows of $W_n$ having the minimum number, say $m_1>0$, of vectors $z_j$ with $|z_j|=\lambda$. Among all such optimal flows, we choose $\varphi$ in such a way that it has the minimum number, say $m_2$, of points $p_j$ with $|p_j|=\lambda$, that is with the minimum number of points which belong to $\mathcal{C}_E$. 
 
\begin{claim}
\label{claim:old_123}
$|z_j|,|p_j| \in \{1,\lambda\}$ for every $j \in \{0,\dots,n-1\}$. In particular, $|z_j|=\lambda$ if and only if $z_j$ is positive (and then  $|z_j|=1$ if and only if $z_j$ is negative).
\end{claim}
\begin{proof}
First we prove that if $z_j$ is positive then $|z_j|=\lambda$. By contradiction suppose there exists $h \in \{0,\dots,n-1\}$ such that $|z_h|<\lambda$ and $z_h$ is positive.
By Claim \ref{remark:transformations} we can shorten $z_k$ and lengthen $z_h$ yielding an optimal flow having less than $m_1$ vectors with modulo $\lambda$, a contradiction.
%

In a similar way we prove that if $z_j$ is negative then $|z_j|=1$. By contradiction assume there exists $h \in \{0,\dots,n-1\}$ such that $|z_h|>1$ and $z_h$ is negative. Following Claim \ref{remark:transformations} we shorten $z_k$ and $z_h$, obtaining again a contradiction as in the previous case on the choice of $\varphi$.


Hence, for every $z_j$ we have $|z_j|=\lambda$ if $z_j$ is positive, while $|z_j|=1$ if $z_j$ is negative.

We complete the proof of the claim by showing that there is no index $h$ such that $1<|p_h|<\lambda$. If this is the case, then we will construct a new sequence of points $p_j'$ by applying a suitable transformation of the original sequence which leads to a contradiction. If $z_h$ and $z_{h+1}$ are both positive, we set $p_h'=(1-\varepsilon)p_h$ and $p_j'=p_j$ for all $j \neq h$, where $\varepsilon>0$ is chosen sufficiently small in such a way that $1<|p_{h-1}'p_h'|<\lambda$ and $1<|p_{h}'p_{h+1}'|<\lambda$. The new sequence corresponds to an optimal flow with $m_1-2$ vectors  with modulo $\lambda$, a contradiction.
If $z_h$ and $z_{h+1}$ are both negative, we set $p_h'=(1+\varepsilon)p_h$ and $p_j'=p_j$ for all $j \neq h$, where $\varepsilon>0$ is chosen sufficiently small in such a way that $1<|p_{h-1}'p_h'|<\lambda$ and $1<|p_{h}'p_{h+1}'|<\lambda$.
Then we shorten $z_k$ and $z_h' = p_{h-1}'p_h'$ as in Claim \ref{remark:transformations} and we obtain an optimal flow with $m_1-1$ vectors with modulo $\lambda$, a contradiction.
If $z_h$ is positive and $z_{h+1}$ is negative, then we trasform the original sequence by using $\sigma_{h,h+1}(\theta)$, where $\theta$ is sufficiently small and it is positive (resp.\ negative) if the angle $p_{h-1}\hat{p_h}p_{h+1}$ is non-negative (resp.\ negative).
Vice versa, if $z_h$ is negative and $z_{h+1}$ is positive, then we transform the original sequence by using $\sigma_{h,h-1}(\theta)$, where $\theta$ is sufficiently small and it is negative (resp.\ positive) if the angle $p_{h-1}\hat{p_h}p_{h+1}$ is non-negative (resp.\ negative).
In all cases, the resulting sequence of points defines a complex nowhere-zero flow on $W_n$ with less than $m_1$ vectors having modulo $\lambda$, a contradiction. This completes the proof of Claim \ref{claim:old_123}.
 \end{proof}
	
By Claim \ref{claim:old_123} we have only eight different types of vectors $z_j$ in $\varphi$. Indeed, $z_j=p_{j-1}p_j$ is completely defined up to rotations once we have its direction (and then its modulo by Claim \ref{claim:old_123}) and the modulo of $p_{j-1}$ and $p_j$ is in $\{1,\lambda\}$. Then, a vector $z_j$ can be denoted by $XY^*$(see Figure \ref{fig:8_types}), where $X,Y \in \{I,E\}$ and $* \in \{+,-\}$ are chosen in the following way.
\begin{itemize}
\item[$\bullet$] $X=I$ if $|p_{j-1}|=1$ and $X=E$ if $|p_{j-1}|=\lambda$;
\item[$\bullet$] $Y=I$ if $|p_j|=1$ and $Y=E$ if $|p_j|=\lambda$;
\item[$\bullet$] $*=+$ or $*=-$ if $z_j$ is positive or negative, respectively.
\end{itemize}	
	
\begin{figure}[h]
\center
\includegraphics[]{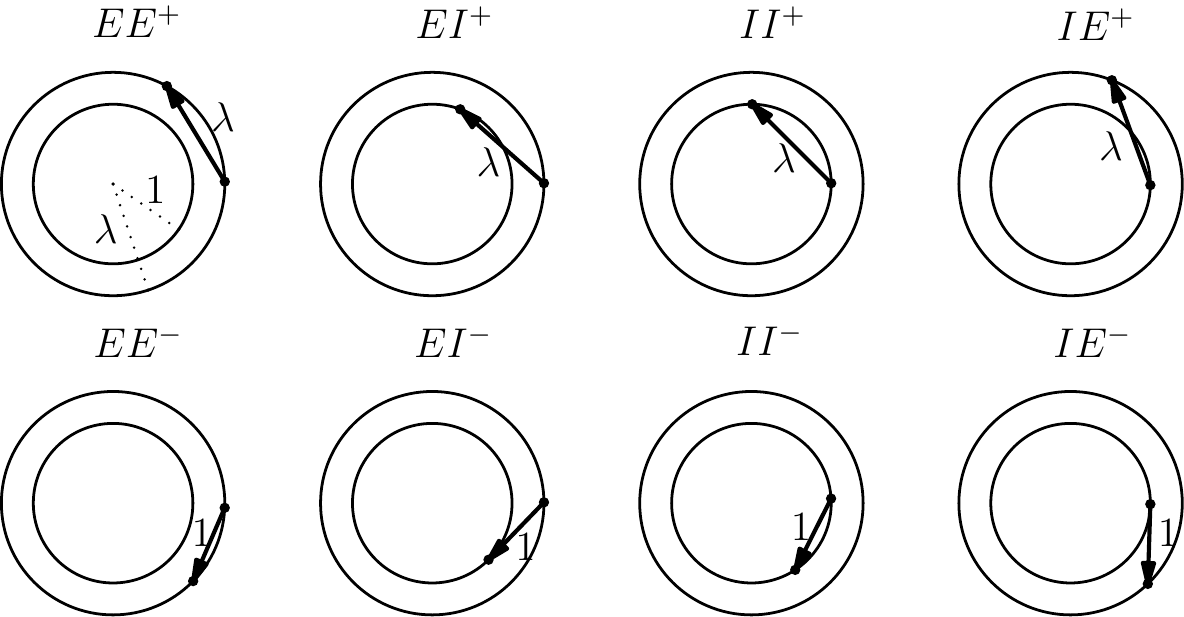}
\caption{The eight types of vectors $z_j$ in a representation of the chosen optimal flow $\varphi$ of $W_n$.}
\label{fig:8_types}
\end{figure}	
	
Without loss of generality we can assume that $|p_0|=1$. Indeed, if $|p_j|=\lambda$ for every index $j$, then all vectors $z_j$ are of type either $EE^+$ or $EE^-$. 
Moreover, $z_k$ is of type $EE^+$ and at least one of them is of type $EE^-$. Otherwise, $|p_j|=\lambda$ and $|z_j|=\lambda$ for all $j$, that is impossible in an optimal flow.

In particular, there must be an index $h$ with $z_h$ of type $EE^+$ and $z_{h+1}$ of type $EE^-$. The sequence is cyclic so we surely find the sequence $EE^+, EE^-$. Hence we can construct a sequence of points $p_j'$ which defines a flow with less than $m_1$ vectors of maximum length $\lambda$ by applying the transformation $\sigma_{h, h+1}(\theta)$, for a sufficiently small $\theta >0$. 


\begin{claim}
All positive vectors $z_j$ are of type $II^+$.
\label{claim:positive_vectors}
\end{claim}
\begin{proof}
First we prove that $\varphi$ has no vector $z_j$ of type $IE^+$. By contradiction, assume $z_h$ of type $IE^+$. There are four possibilities for the vector $z_{h+1}$, namely $EE^+$, $EE^-$, $EI^+$ and $EI^-$. Since we have $\lambda \leq \sqrt{2}<\Phi$ for every odd $n$, the mutual position of the two vectors $z_h$ and $z_{h+1}$ in each case is like the ones represented in Figure \ref{fig:no_IEplus}.
\begin{figure}[h]
\center
\includegraphics[]{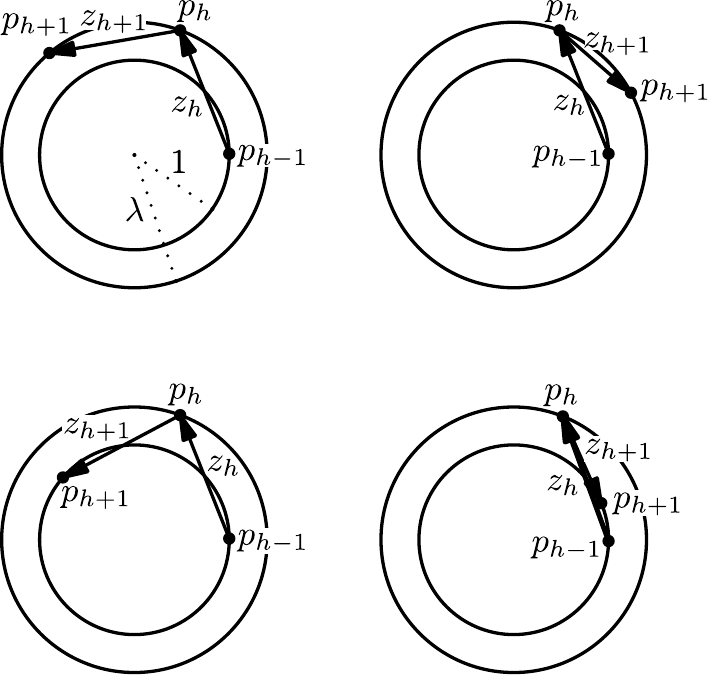}
\caption{Mutual position of $z_h$ and $z_{h+1}$.}
\label{fig:no_IEplus}
\end{figure}
In all these cases, by applying $\sigma_{h,h-1}(\theta)$ for a sufficiently small $\theta>0$ we obtain a new sequence of points $p_j'$ which corresponds to an optimal flow with either less than $m_1$ vectors of norm $\lambda$ (if $z_{h+1}$ is positive) or  $m_1$ vectors of norm $\lambda$ but less than $m_2$ points on $\mathcal{C}_E$ (if $z_{h+1}$ is negative), a contradiction in both cases.

Moreover, $\varphi$ has no vector of type $EE^+$. Indeed, note that the angle subtended at the centre by a vector of type $EE^+$ on $\mathcal{C}_E$ is equal to $\frac{\pi}{3}$. Then, it is the same angle subtended at the centre  by a vector of type $II^+$ on $\mathcal{C}_I$. 
If $p_{h-1}p_h$ is of type $EE^+$ (note that $h\geq2$ since $p_0 \in \mathcal{C}_I$), then we can construct a new sequence of points $(p_0',...,p_{n-1}')$ in the following way:

\begin{itemize}
\item[-] $p_0'=p_0$
\item[-] $|p_0'p_1'|=1$ and $p_1' \in \mathcal{C}_I$
\item[-] $p_{j-1}'p_{j}'$ is of the same type of $p_{j-2}p_{j-1}$ for $2\leq j\leq h$
\item[-] $p_{j-1}'p_{j}'$ is of the same type of $p_{j-1}p_{j}$ for $h <j\leq n-1$.
\end{itemize}
Since we replaced a vector $p_{h-1}p_{h}$ of type $EE^+$ with a vector $p_0'p_1'$ of type $II^+$ which subtends the same angle, while mantaining all the other vectors of the same type, the new sequence of points has less than $m_1$ vectors having norm $\lambda$, a contradiction.

Finally, we prove that $\varphi$ has no vectors of type $EI^+$. Indeed, if $p_{h-1}p_h$ is of type $EI^+$ with $h>0$, then $p_{h-2}p_{h-1}$ cannot be positive because both vectors of type $EE^+$ and $IE^+$ are already excluded. Then, it could be either of type $IE^-$ or $EE^-$. Since $\lambda \leq \sqrt{2}$ the mutual position of the points $p_{h-2},p_{h-1}$ and $p_h$ is like the ones in Figure \ref{fig:no_EIplus}.
\begin{figure}[h]
\center
\includegraphics[scale=1.3]{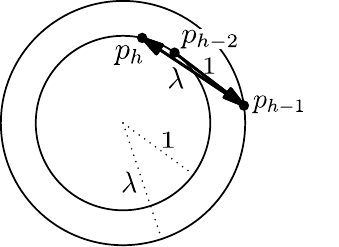}
\qquad
\includegraphics[scale=1.3]{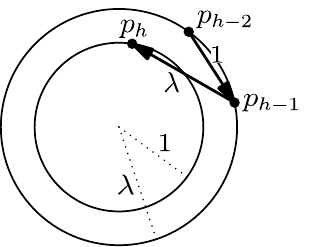}
\caption{Mutual position of the points $p_{h-2},p_{h-1}$ and $p_h$.}
\label{fig:no_EIplus}
\end{figure}

In both these cases, by applying $\sigma_{h-1,h-2}(\theta)$ for a sufficiently small $\theta <0$, we obtain a new configuration of points $p_j'$ which corresponds to an optimal flow with less than $m_1$ vectors of norm $\lambda$, a contradiction. This completes the proof of Claim 2.
\end{proof}

In what follows we will make use of the measure of some angles depicted in Figure \ref{fig:useful_angles}. We denote by $2\alpha$ and $2\beta$ the angles subtended  at the centre by a chord of length $1$ on $\mathcal{C}_E$ and of length $\lambda$ on $\mathcal{C}_I$, respectively. The following relations hold.
$$\alpha=\arcsin\left(\frac{1}{2\lambda}\right), \beta=\arcsin\left(\frac{\lambda}{2}\right).$$

\begin{figure}[h]
\center
\includegraphics[]{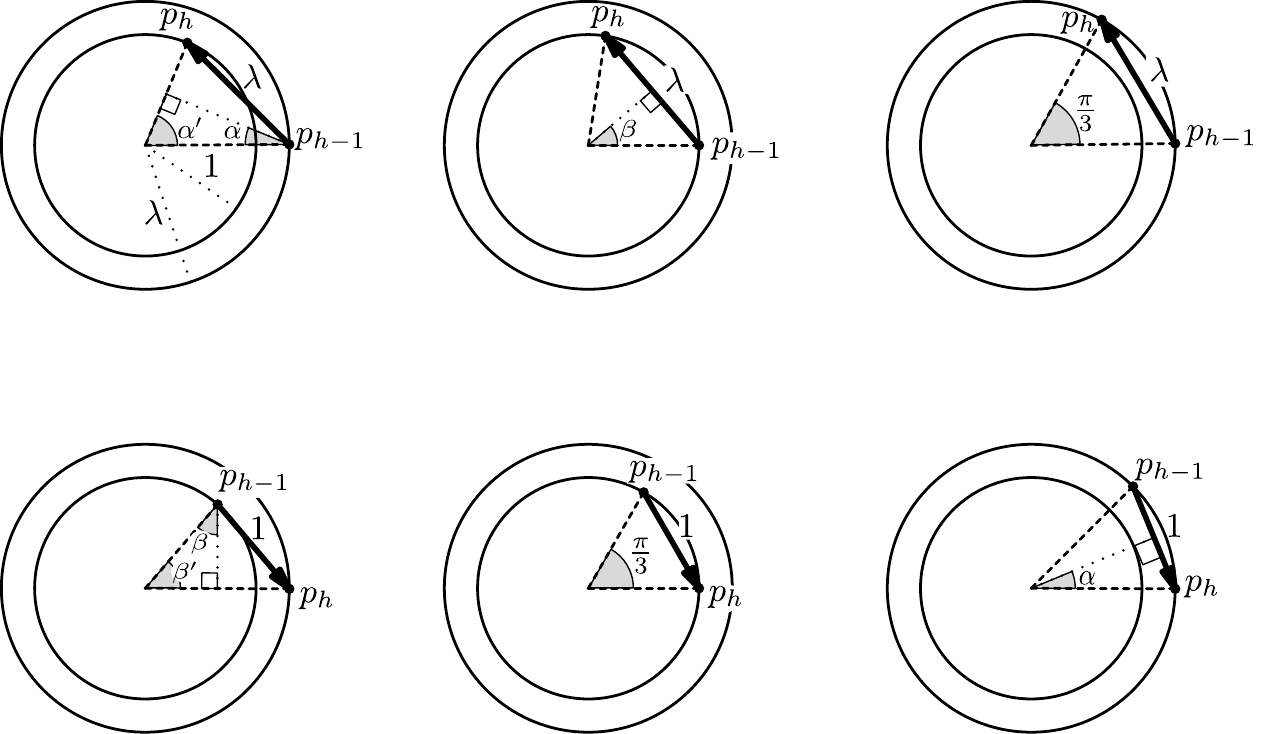}
\caption{The angles subtended at the centre by all different types of vectors having norm $1$ and $\lambda$.}
\label{fig:useful_angles}
\end{figure}

Since $1<\lambda \leq \sqrt{2}$, it follows $$\arcsin{(\frac{\sqrt{2}}{4})}\leq \alpha<\frac{\pi}{6},\frac{\pi}{6}<\beta\leq \frac{\pi}{4}.$$

Moreover, we prove the inequality $\alpha+\beta>\frac{\pi}{3}$ which will be used in what follows.
$$ \cos(\alpha+\beta)
=\cos \alpha \cos \beta - \sin \alpha \sin \beta = \sqrt{\left(1-\frac{1}{4\lambda^2} \right)\left(1-\frac{\lambda^2}{4} \right)} - \frac{1}{4}.$$

Since $\alpha+\beta$ is not larger than $\frac{5\pi}{3}$, then $\alpha+\beta > \frac{\pi}{3}$ if and only if $\sqrt{\left(1-\frac{1}{4\lambda^2} \right)\left(1-\frac{\lambda^2}{4} \right)} - \frac{1}{4}<\frac{1}{2}$. This inequality easily leads to  $4(\lambda^2-1)^2>0$ which is always satisfied. 

Finally, we denote by $\alpha'= \frac{\pi}{2}-\alpha$ and $\beta'= \frac{\pi}{2}-\beta$ the complement angles of $\alpha$ and $\beta$, respectively.

\begin{claim}
No vector $z_j$ is of type $EE^-$.
\label{claim:noEEminus}
\end{claim}
\begin{proof}
By Claim \ref{claim:positive_vectors} and since $p_0 \in \mathcal{C}_I$, to prove this claim it suffices to show that the two ordered sequences of three consecutive vectors of types $IE^-, EE^-, EI^-$ and $IE^-, EE^-,EE^-$ cannot appear in $\varphi$.

We first prove that a subsequence of type $IE^-, EE^-, EI^-$ cannot appear. 
Assume that the points $p_j$ corresponding to the subsequence $IE^-, EE^-, EI^-$ are $p_j,p_{j+1},p_{j+2}$ and $p_{j+3}$ as in Figure \ref{fig:no_EEminus_a}.
Observe that the angle subtended at the centre by $p_jp_{j+3}$ is  
\begin{equation}
\beta'+2\alpha+\beta'=2\beta'+2\alpha=2\left(\frac{\pi}{2}-\beta\right)+2\alpha=\pi-2(\beta-\alpha)<\pi
\label{eq:1}
\end{equation}
 where the last inequality holds since $\beta>\alpha$ for every $\lambda \in [1,\sqrt{2}]$.
 
\begin{figure}[h]
\center
\includegraphics[]{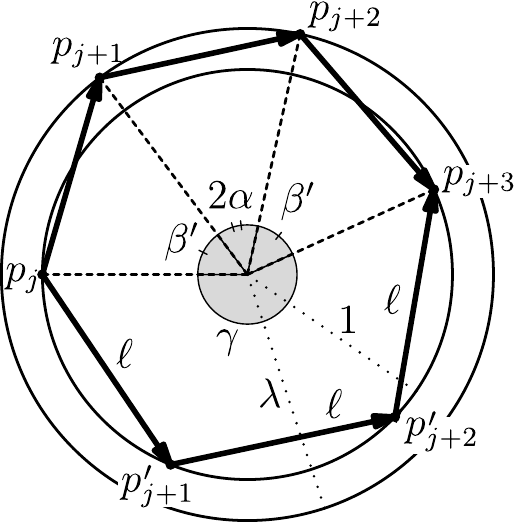}
\caption{Configuration of points $p_j,p_{j+1},p_{j+2}$ and $p_{j+3}$ corresponding to the subsequence of types $IE^-, EE^-, EI^-$.}
\label{fig:no_EEminus_a}
\end{figure}

Replace $p_{j+1}$ and $p_{j+2}$ by $p_{j+1}', p_{j+2}' \in \mathcal{C}_I$ in such a way that $p_jp_{j+1}', p_{j+1}'p_{j+2}'$ and $p_{j+2}'p_{j+3}$ are all positive vectors with $|p_jp_{j+1}'|=|p_{j+1}'p_{j+2}'|=|p_{j+2}p_{j+3}'|=\ell$ (see Figure \ref{fig:no_EEminus_a}).
Let us prove that $1<\ell <\lambda$. Indeed,  we have
$$\gamma=2\pi-(2\alpha+2\beta')=\pi-2\alpha+2\beta$$
Recalling that $\alpha+\beta>\frac{\pi}{3}$ and $\beta>\frac{\pi}{6}$, we obtain $2\alpha+4\beta>\pi$, that is $6\beta> \pi-2\alpha+2\beta=\gamma$ and so $\frac{\gamma}{3}<2\beta$. Hence, $\ell<\lambda$.
Moreover, since $\gamma >\pi$, and so $\frac{\gamma}{3}>\frac{\pi}{3}$, by (\ref{eq:1}), we have also that $\ell>1$.
Hence, the new sequence of points corresponds to an optimal flow with the same number $m_1$ of vectors of norm $\lambda$, but less than $m_2$ points on $\mathcal{C}_E$, a contradiction.

In a very similar way we prove that the subsequence of types $IE^-, EE^-,EE^-$ cannot appear in a representation of $\varphi$. Again, assume that the points $p_j$ corresponding to the subsequence $IE^-, EE^-, EE^-$ are $p_j,p_{j+1},p_{j+2}$ and $p_{j+3}$.
\begin{figure}[h]
\center
\includegraphics[]{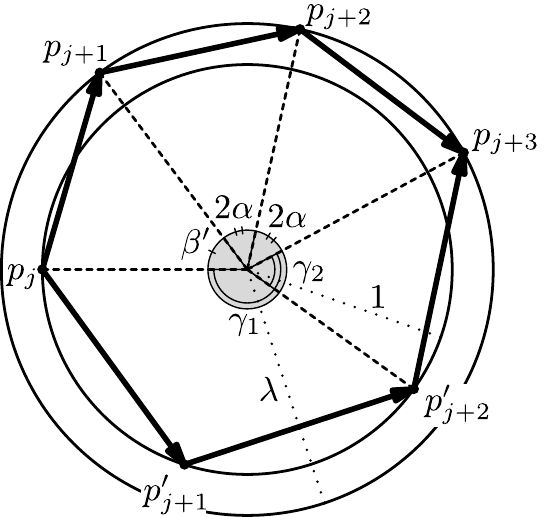}
\caption{Configuration of points $p_j,p_{j+1},p_{j+2}$ and $p_{j+3}$ corresponding to the subsequence of types $IE^-, EE^-,EE^-$.}
\label{fig:no_EEminus_b}
\end{figure}

Again we replace the points $p_{j+1}$ and $p_{j+2}$ of the sequence by two new points $p_{j+1}',p_{j+2}' \in \mathcal{C}_I$ in such a way that $|p_jp_{j+1}'|=|p_{j+1}'p_{j+2}'|=|p_{j+2}'p_{j+3}|= \ell$, as shown in Figure \ref{fig:no_EEminus_b}. We prove that $1<\ell<\lambda$. 
Denote by $\gamma_1$ the angle subtended at the centre by $p_jp_{j+2}'$ and by $\gamma_2$ the angle subtended by $p_{j+2}'p_{j+3}$.
Then, $\ell$ is such that $\beta'+4\alpha+\gamma_1+\gamma_2=2\pi$ holds. 

In order to prove $\ell<\lambda$, it suffices to show that the sum of the angles obtained with $|p_jp_{j+1}'|=|p_{j+1}'p_{j+2}'|=|p_{j+2}'p_{j+3}|= \lambda$, that is $\beta'+4\alpha+4\beta+\alpha'$, is strictly larger than $2\pi$. Since 

$$\beta'+4\alpha+4\beta+\alpha'= \frac{\pi}{2}-\beta+4\alpha+4\beta+\frac{\pi}{2}-\alpha= \pi+3(\alpha+\beta),$$

it follows that $\beta'+4\alpha+4\beta+\alpha'>2\pi$ if and only if $\alpha+\beta>\frac{\pi}{3}$, which is already proved to be satisfied.

In order to prove $\ell>1$, it suffices to show that the sum of the angles obtained with $|p_jp_{j+1}'|=|p_{j+1}'p_{j+2}'|=|p_{j+2}'p_{j+3}|=1$, that is $\beta'+4\alpha+2\frac{\pi}{3}+\beta'$, is strictly smaller than $2\pi$. Since

%

$$\beta'+4\alpha+2\frac{\pi}{3}+\beta'= \pi-2\beta+4\alpha+\frac{2}{3}\pi,$$

it follows that $\beta'+4\alpha+2\frac{\pi}{3}+\beta'<2\pi$ if and only if $2\alpha-\beta<\frac{\pi}{6}$.

Recalling that $\alpha+\beta>\frac{\pi}{3}$ and $\alpha<\frac{\pi}{6}$, we have $2\alpha-\beta=3\alpha-(\alpha+\beta) < \frac{\pi}{2} -\frac{\pi}{3} = \frac{\pi}{6}$.


Hence, the new sequence of points corresponds to an optimal flow having $m_1-1$ vectors of norm $\lambda$, a contradiction once again. This completes the proof of Claim 3.

\end{proof}

Now we permute the sequence of points $(p_0,...,p_{n-1})$ associated to $\varphi$ in a sequence denoted by $(q_0,...,q_{n-1})$, in such a way that $(w_0,...,w_{n-1})$ is also a permutation of the vectors $(z_0,...,z_{n-1})$, where $z_j=p_{j-1}p_j$ and $w_j=q_{j-1}q_j$, up to rotations of each vector around the origin. We mean that for every vector $z_h=p_{h-1}p_h$ there exists a vector $w_k=q_{k-1}q_k$ such that $q_k$ and $q_{k-1}$ are obtained by a suitable rotation of the points $p_h$ and $p_{h-1}$ of the same angle around the origin.

By definition, the sequence of types of the vectors $w_j$ is a permutation of the sequence of the types of the vectors $z_j$. Hence, the values of $m_1$ and $m_2$ do not change for this new sequence.

By previous claims such a sequence can contain only the following four types of vectors: $IE^-,EI^-,II^-$ and $II^+$.  
Moreover, if a vector of type $IE^-$ appears in the sequence, then it is necessarily followed by a vector of type $EI^-$.
We choose $(q_0,...,q_{n-1})$ in such a way that $p_0 \equiv q_0$ and all pairs $IE^-,EI^-$, if present, appear at the beginning of the sequence. They are followed by all vectors of type $II^-$, if present, and finally by all vectors of type $II^+$. The sequence of types can be described in general by the following ordered sequence.
$$(IE^-,EI^-,...,IE^-,EI^-,II^-,\dots,II^-,II^+,...,II^+)$$

Now we prove that some specific subsequences cannot appear in the sequence associated to $(q_0,\dots,q_{n-1})$.

\begin{claim}\label{claim:no_3_subsequences}
The subsequences of consecutive vectors of types
\begin{itemize}
\item[(a)] $IE^-,EI^-,IE^-$
\item[(b)] $IE^-EI^-,II^-$
\item[(c)] $II^-,II^-$
\end{itemize}  
cannot appear in the ordered sequence of types associated to $(q_0,...,q_n).$
\end{claim}
\begin{proof}
\textit{(a)} We argue similarly to what we did in the proof of Claim \ref{claim:noEEminus}. Assume that the four consecutive points corresponding to the subsequence $IE^-, EI^-, IE^-$ are $q_j,q_{j+1},q_{j+2}$ and $q_{j+3}$.
We obtain a new sequence by replacing the two points $q_{j+1},q_{j+2}$ by the points $q_{j+1}',q_{j+2}' \in \mathcal{C}_I$ such that $|q_jq_{j+1}'|=|q_{j+1}'q_{j+2}'|$ and $1<|q_{j+2}'q_{j+3}|<\lambda$ (see Figure \ref{fig:no_sequence_a}). Let us show that such a choice is admissible.

\begin{figure}[h]
\center
\includegraphics[]{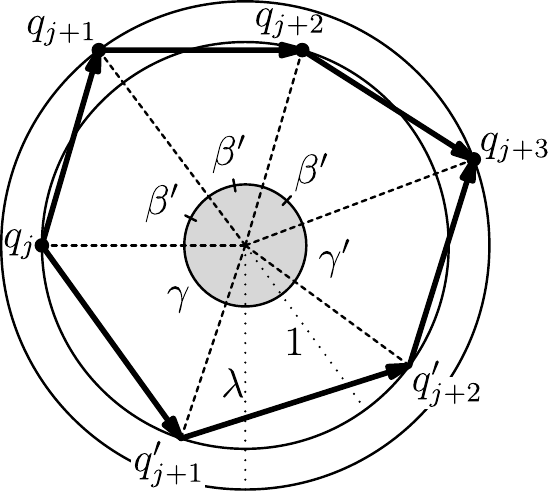}
\caption{Configuration of points $q_j,q_{j+1},q_{j+2}$ and $q_{j+3}$ corresponding to the subsequence of types $IE^-, EI^-, IE^-$.}
\label{fig:no_sequence_a}
\end{figure}

Denote by $\gamma$ and $\gamma'$ the angles subtended at the centre by $q_jq_{j+2}'$ and $q_{j+2}'q_{j+3}$, respectively. Since $|p_{j+2}'p_{j+3}|>1$, we have that $\gamma'>\beta'$.
Hence, $\gamma=2\pi-3\beta'-\gamma'<2\pi-4\beta'=4\arcsin(\frac{\lambda}{2})=4\beta$. 

Since $\beta' < \pi/3$, $2\pi - 4\beta' > \frac{2}{3}\pi$. Note that, $\gamma'= \beta' + \varepsilon$, for a certain $\varepsilon>0$. We choose $\varepsilon$ sufficiently small in such a way that $\gamma = 2\pi - 3\beta' -\gamma' = 2\pi - 4\beta' - \varepsilon > \frac{2}{3}\pi$. So $|q_jq_{j+1}'|>1$. Moreover, $\gamma = 2\pi - 3\beta' - \gamma' < 2\pi - 4\beta' = 4\beta$ implies that $|q_jq_{j+1}'|< \lambda$.

The new sequence of points has $m_1$ vectors of norm $\lambda$, but less than $m_2$ points which belongs to $\mathcal{C}_E$, a contradiction.

\textit{(b)} Assume that the four consecutive points corresponding to the subsequence $IE^-, EI^-, II^-$ are $q_j,q_{j+1},q_{j+2}$ and $q_{j+3}$. Denote by $\gamma$ the explement angle of the angle subtended at the centre by $q_jq_{j+3}$.  We obtain a new sequence by replacing the two points $q_{j+1},q_{j+2}$ by the points $q_{j+1}',q_{j+2}' \in \mathcal{C}_I$ such that $|q_jq_{j+1}'|=|q_{j+1}'q_{j+2}'|=|q_{j+2}'q_{j+3}|$  (see Figure \ref{fig:no_sequence_b}).
Since $\beta'<\frac{\pi}{3}$, we have $2\beta'+\frac{\pi}{3}<\pi$, then $|q_jq_{j+1}'|>1$. 

Moreover, since $\lambda > 1$, $\gamma 
         = 2\pi - (\frac{\pi}{3} + 2\beta') = \frac{2}{3}\pi + 2\beta < 6\beta$. Hence, $|q_jq_{j+1}'| < \lambda$.

The new sequence of points corresponds to an optimal flow having $m_1$ vectors of maximum norm $\lambda$, but less than $m_2$ points $q_i$ on $\mathcal{C}_E$, a contradiction.

\begin{figure}[h]
\center
\includegraphics[]{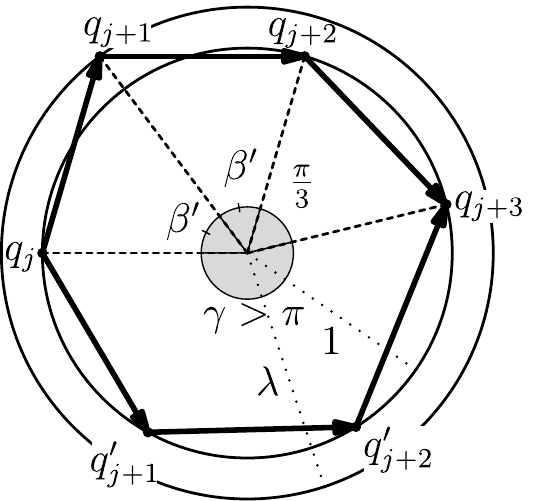}
\caption{Configuration of points $q_j,q_{j+1},q_{j+2}$ and $q_{j+3}$ corresponding to the subsequence of types $IE^-, EI^-, II^-$.}
\label{fig:no_sequence_b}
\end{figure}

\textit{(c)} Let $w_{j}$ and $w_{j+1}$ be the last two vectors of type $II^-$ in the sequence, that is $w_{j+2}$ is of type $II^+$. 
We obtain a new sequence by replacing the two points $q_{j+1},q_{j+2} \in \mathcal{C}_I$ by the points $q_{j+1}',q_{j+2}'$ such that $q_{j+1}'=(1+\varepsilon)q_{j+1}$ and $q_{j+2}'=q_{j+2}e^{i\theta}$, where $\varepsilon>0$ and $\theta>0$ are chosen sufficiently small and in such a way that $|q_{j+1}'q_{j+2}'|=1$ and $1<|q_{j+2}'q_{j+3}|<\lambda$.
The new sequence of points corresponds to an optimal flow having less than $m_1$ vectors of maximum norm $\lambda$, a contradiction.
This completes the proof of Claim 4.
\end{proof}

By previous claims there exists an optimal flow of $W_n$ such that the types of its vectors respect one of the following three sequences.
\begin{itemize}
\item[(i)] $(IE^-, EI^-, II^+, ..., II^+)$
\item[(ii)] $(II^-, II^+,...,II^+)$
\item[(iii)] $(II^+,..., II^+)$ 
\end{itemize}

For each given $n$ odd, the value of $\lambda$ is completely determined once we know which of the three sequences we are considering. The exact values of $\lambda$ in each of these cases are summarized in Table \ref{Table_nine_cases}. We give here an example of direct computation of the values in the last column of the table. The remaining values are computed similarly. In this particular sequence, every vector has norm $\lambda$, while all points $p_j$ belong to $\mathcal{C}_I$. Hence, the angle subtended at the centre by each $z_j$ on $\mathcal{C}_I$ is exactly $2\beta=2\arcsin (\frac{\lambda}{2})$. Hence, for some integer $k>0$, we have $2n\arcsin \frac{\lambda}{2}=2k\pi$, that is

$$\lambda=2\sin\left( \frac{k}{n}\pi \right) .$$

We are looking for the minimum possible $\lambda > 1$. Then, $k$ is chosen as the smallest integer such that $\frac{k}{n} \pi > \frac{\pi}{6}$, that is $k = \lceil \frac{n}{6}\rceil$. 
Set $n=6t+h$, for $h=1,3,5$. We obtain $\lambda= 2\sin\left(\frac{\pi(t+1)}{6t+h}\right)=
2\sin\left(\frac{\pi}{6}\cdot\frac{6t+6}{6t+h}\right)=2\sin\left(\frac{\pi}{6}\cdot\frac{n+(6-h)}{n}\right).$

Finally, comparing for each congruence of $n$ the three possible values for $\lambda$, it turns out that the minimum $\lambda$ is obtained with configuration (ii) if $n \equiv 1\mod 6$, configuration (i) if $n \equiv 3\mod 6$ and configuration (iii) if $n \equiv 5\mod 6 $. The statement follows.


	\begin{table}[h]
\centering
\begin{tabular}{cccc}
  &$IE^-,EI^-,II^+,...,II^+$ & $II^-,II^+,...,II^+$ &  $II^+,...,II^+$ \\
\cmidrule[1pt]{1-4}
$n \equiv 1\mod 6$ & $2\sin\bigl(\frac{\pi}{6}\cdot\frac{n+2}{n-1}\bigr)$  & \boldmath{$2\sin \bigl(\frac{\pi}{6}\cdot \frac{n}{n-1}\bigr)$} & $2\sin\bigl(\frac{\pi}{6}\cdot\frac{n+5}{n}\bigr)$  \\
\cmidrule[0.5pt]{1-4}
$n \equiv 3\mod 6$& \boldmath{$2\sin\bigl( \frac{\pi}{6}\cdot\frac{n}{n-1}\bigr)$}  & $2\sin \bigl(\frac{\pi}{6}\cdot \frac{n+4}{n-1}\bigr)$  & $2\sin\bigl(\frac{\pi}{6}\cdot\frac{n+3}{n}\bigr)$  
\\
\cmidrule[0.5pt]{1-4}
$n \equiv 5\mod 6$ & $2\sin\bigl(\frac{\pi}{6} \cdot \frac{n+4}{n-1}\bigr)$  & $2\sin \bigl(\frac{\pi}{6} \cdot \frac{n+2}{n-1}\bigr)$ & \boldmath{$2\sin\bigl(\frac{\pi}{6}\cdot\frac{n+1}{n}\bigr)$}  
\\
\cmidrule[1pt]{1-4}
\end{tabular}
\caption{Exact values for $\lambda$ in configurations $(i),(ii)$ and $(iii)$, according to the congruence of $n$ modulo $6$. In bold the minimum value for each of the three cases.}\label{Table_nine_cases}
\end{table}

\end{itemize}
\end{proof}

\section{A general lower bound for $\phi_{\mathbb{C}}$}
The value $\phi_\mathbb{C}(W_n)$, for each odd $n$, gives a general non-trivial lower bound for $\phi_\mathbb{C}(G)$, where $G$ is a bridgeless cubic graph, in terms of its odd-girth (the length of a shortest odd cycle of $G$). This is a straightforward consequence of the following standard observation. If $C$ is an odd cycle of minimum length in $G$, then $C$ is chordless. Contract all vertices of $G$ not belonging to $C$ to a unique vertex, thus obtaining a wheel graph whose complex flow number cannot be more than the complex flow number of $G$. Then, by Theorem \ref{thm:wheels} we deduce the following general result.
 
\begin{theorem}\label{cor:oddgirth}
Let $G$ be a non-bipartite cubic graph and let $g$ be its odd-girth. Then,
$$\phi_{\mathbb{C}}(G) \geq \begin{cases}
	1 + 2\sin(\frac{\pi}{6}\cdot\frac{g}{g-1})& \text{ if } g \equiv 1,3\mod 6, \\
	1 + 2\sin(\frac{\pi}{6}\cdot\frac{g+1}{g})& \text{ if } g\equiv 5 \mod 6. \\                                                                                                                               \end{cases}
	$$
\end{theorem}

Let us remark that lower bounds in Theorem \ref{cor:oddgirth} are tight due to the prism graph $P_n$ of order $2n$. Indeed, it is easy to see that each complex $\phi_{\mathbb{C}}(W_n)$-flow on $W_n$ can be extended to a complex $\phi_{\mathbb{C}}(W_n)$-flow on $P_n$ by a symmetry argument. Hence, the following holds.

\begin{corollary}\label{cor:prisms}
Let $P_n$ be the prism graph of order $2n$, $n \geq 3$. Then,
$$\phi_{\mathbb{C}}(P_n) = \begin{cases}
    2 & \text{ if } n \text { even,}\\ 

	1 + 2\sin(\frac{\pi}{6}\cdot\frac{n}{n-1})& \text{ if } n \equiv 1,3\mod 6, \\
	1 + 2\sin(\frac{\pi}{6}\cdot\frac{n+1}{n})& \text{ if } n\equiv 5 \mod 6. \\                                                                                                                               \end{cases}
	$$
\end{corollary}

\section{Acknowledgments}

Davide Mattiolo is supported by a Postdoctoral Fellowship of the Research Foundation Flanders (FWO), project number 1268323N.

\end{document}